\documentclass[a4paper,reqno,11pt,final]{article}
\usepackage{latexsym}
\usepackage{amsthm,amsmath,amssymb}

\usepackage{graphicx, color}

\usepackage{transparent}

\usepackage{tikz}
\usetikzlibrary{arrows}
\usepackage{rotating}
\usepackage{lscape}

\newtheorem{thm}{Theorem}

\newtheorem{defn}{Definition}

\newtheorem{prop}[thm]{Proposition}

\newtheorem{lem}[thm]{Lemma}

\newcommand{\Aut}{\mathrm{Aut}}

\newcommand{\Hom}{\mathrm{Hom}}

\newcommand{\F}{\mathbb {F}}

\def\={\;=\;} \def\+{\,+\,}       \def\Q{\Bbb Q} \def\O{\mathcal O}  \def\C{\Bbb C}  \def\R{\Bbb R}  
        
  \def\Aut{\text{Aut}}    
  \def\Z{\mathbb Z}

\def\K{\mathcal K}

\def\F{\mathcal F}
\def\L{\mathcal L}
\def\O{{\mathcal O}}
\def\Ar{\!\!\mathrm{ar}}
\def\ar{\mathrm{ar}}
\def\num{\mathrm{num}}
\def\Num{\mathrm{Num}}
\def\di{\mathrm{div}}

\def\qqan{\qquad\mathrm{and}\qquad}
\def\dis{\displaystyle}
\def\ov{\overline}
\def\GL{\text{GL}(V,A)}

\def\raw{\longrightarrow}
\def\law{\longleftarrow}

\newcommand*{\mydprime}{^{\prime\prime}\mkern-1.2mu}

\begin{document}
\title{\bf {\Large{Arithmetic Central Extensions\\ and \\Reciprocity Laws for Arithmetic Surface}}}
\author{\bf  K. Sugahara and L. Weng}
\date{}
\maketitle

\begin{abstract} Three types of reciprocity laws for arithmetic surfaces are established.  
For these around a point or along a vertical curve, 
we first construct $K_2$ type central extensions, then  introduce reciprocity symbols, 
and finally prove the law as an application of Parshin-Beilinson's theory of adelic complex. 
For reciprocity law along a horizontal curve, we first introduce  a new  type 
of arithmetic central extensions,  then apply our arithmetic adelic cohomology theory 
and  arithmetic intersection theory to prove the related reciprocity law.

All this can be interpreted within the framework of arithmetic central extensions. We add an appendix
to deal with some basic structures of such extensions.
\end{abstract}
\tableofcontents
\section{Introduction}

About 50 years ago, Tate \cite{T} developed a theory of residues for curves 
using traces and adelic cohomologies.  Tate's work was 
integrated 
with the $K_2$ central extensions by Arbarello-De Concini-Kac \cite{ADCK}. 
Reciprocity laws for algebraic surfaces were established by Parshin in \cite{P2}.
Later, these reciprocity laws were reproved by Osipov \cite{O1}, 
based on Kapranov's dimension theory \cite{K}. Osipov constructs
dimension two central extensions and hence establishes the reciprocity 
law 
for algebraic surfaces using Parshin's adelic theory \cite{P1}. More recently, a
categorical proof of Parshin's reciprocity law was found by Osipov-Zhu \cite{OZ}.
In essence, Osipov's 
construction may be viewed as $K_2$ type theory of central extensions
developed by Brylinski-Deligne \cite{BD}. 

However, to establish reciprocity laws for arithmetic surfaces,
algebraic $K_2$ theory of central extensions is not sufficient: while it works for 
reciprocity laws around points and along vertical curves, it fails when we treat horizontal curves.
To remedy this, instead, we first develop a new theory of arithmetic central extensions, 
based on the fact that for exact sequence of metrized $\R$-vector saces of finite dimensional
$$\overline V_*:\qquad 0\to \overline V_1\to \overline V_2\to \overline V_3\to 0,$$
there is a volume discrepancy $\gamma(\overline V_*)$.
Accordingly, for an $\R$-vector space $V$, not necessary to be finite dimensional, 
a subspace $A$, and commensurable subspaces $B,\,C$, following \cite{ADCK}, we have the group
$$\GL(V,A):=\{g\in\Aut(V)\,:\, A\sim gA\},$$
 the $\R$-line $$(A|B):=\lambda(A/A\cap B)^*\otimes \lambda(B/A\cap B),$$
and the natural contraction isomorphism
$$\alpha:=\alpha_{A,B,C}: (A|B)\otimes (B|C)\simeq (A|C).$$
Note that $A/A\cap B$ and $B/A\cap B$ are finite dimensional $\R$-spaces, we hence can 
introduce metrics on them and hence obtaining the metrized $\R$-line
$\overline{(A|B)}$. In general, the contraction map $\alpha_{A,B,C}$ of (4.3) of  \cite{ADCK} does not give an isometry. We denote
the corresponding discrepancy by $\gamma(\overline{\alpha_{A,B,C}})$ and hence the isometry
$\ov{\alpha}_{A,B,C}:=\alpha_{A,B,C}\cdot \gamma(\overline{\alpha_{A,B,C}})$ so that we obtain 
a canonical isometry
$$\ov \alpha:=\ov \alpha_{A,B,C}: \ov{(A|B)}\otimes \ov{(B|C)}\cong \ov{(A|C)}.$$
In parallel, for a pair of commensurable (metrized) subspaces $A,B$ and $A',B'$, the natural isomorphism $\beta$
introduced in '4.4) of  \cite{ADCK} induced a canonical isometry
$$\ov\beta_{A,B;A',B'}:\ov{(A|B)}\otimes \ov{(A'|B')}\raw \ov{(A\cap A'|B\cap B')}\otimes \ov{(A+A'|B+B')}.$$
The up-shot of this consideration then leads to the following construction of arithmetic central extension 
$\dis{\widehat{\GL}^{\Ar}}$ of $\GL$ is given by
$$ {\widehat{\GL}}^{\Ar}:=\Big\{(g,a)\,:\, g\in\GL, a\in \ov{(A|gA)}, a\not=0 \Big\}$$
together with the multiplication
$$(g,a)\circ(g',a'):=(gg',a\circ g(a'))$$ where 
$$ab:=a\circ b:=\ov\alpha_{A,B,C}(a\otimes b)\in \ov{(A|C)}\qquad\forall a\in \ov{(A|B)},\ b\in \ov{(B|C)}.$$
Our first result is the following analogue of 
\vskip 0.20cm
\noindent
{\bf Proposition A.} 
{\it \begin{itemize}
\item[(1)] $\big({\widehat{\GL}}^{\Ar},\circ,(e,1)\big)$ forms a group;
\item[(2)] There is a canonical central extension of groups
$$1\to \R^*\buildrel\iota\over\raw {\widehat{\GL}}^{\Ar}\buildrel\pi\over \raw \GL\to e.$$
\end{itemize}}


Consequently, for elements $(g,a),\,(h,b)\in {\widehat{\GL}}^{\Ar}$, their commutator is given by
$$[(g,a),\,(h,b)]=\big([g,h],a\circ g(b)\circ (ghg^{-1})(a^{-1})\circ[g,h](b^{-1}).$$
In particular, from Theorem A(2), we have, if $g$ and $h$ commute,
$$\langle g,h\rangle:=\langle g,h\rangle_A:=a\circ g(b)\circ ghg^{-1}(a^{-1})\circ[g,h](b^{-1})\in \R^*.$$
We will show that $\langle g,h\rangle$ is well-defined.
Moreover, we have the following
\vskip 0.20cm
\noindent
{\bf Proposition B.} ({\bf Arithmetic Reciprocity Law}) {\it Let $A,\,B$ be two subspaces of $V$. For 
$g,h\in\GL\cap \mathrm{GL}(V,B)$, we have}
$$
\ov\beta\big(\langle g,h\rangle_A\otimes \langle g,h\rangle_B\big)=\langle g,h\rangle_{A\cap B}\otimes
\langle g,h\rangle_{A+B}.
$$
In particular, if $g$ and $g$ commutes, we have
$$\langle g,h\rangle_A\langle g,h\rangle_B=\langle g,h\rangle_{A\cap B}\langle g,h\rangle_{A+B}.$$

As a special case, if we assume that the metric system is rigid, that is to say,  as to be defined in \S?, we then recover
the constructions of \cite{ADCK}.

To go further, we consider the local field $\R((t))$ of Laurent series with $\R$ coefficients. For each finite dimensional
sub-quotient space of $\R((t))$, we may identify it with a standard form
$V_{m,n}:=\sum_{i=m}^n\R t^i$ for suitable integer $m,\,n\in \Z$.  We assign a metric on $V_{m,n}$ based on 
the standard Eulidean metric of $\R$ with $(t^i,t^j)=\delta_{ij}, \forall m\leq i,j\leq n.$
Then, for any $f(t),\, g(t)\in \R((t))$, if we write $f(t)=t^{\nu_t(f)}f_0(t),\ g(t)=t^{\nu_t(g)}g_0(t)$, then
$\dis{\big(f(t)\,\R[[t]]\,\big|\, g(t)\,\R[[t]]\big)}$ can be explicitly calculated. As a direct consequence,
we have $$\nu_{\R((t))}(f,g):=\log\langle f,g\rangle_{\R[[t]}=\log\frac{\ |f_0(0)^{\nu_t(g)}|\ }{|g_0(0)^{\nu_t(f)}|}.$$

The pairing $\langle f,g\rangle_{\R[[t]}$ is in fact the reciprocity symbol along horizontal curve at infinity. To introduce
reciprocity symbol at finite places, we start with corresponding 2 dimensional local fields $L$, $k_L((u)), k_L\{\{u\}\}$ with $k_L$ finite extensions of $\Q_P$.  On $L$, there exists a discrete valuation of rank 2:
$(\nu_1,\nu_2): L^*\to \mathbb Z\oplus\mathbb Z.$ With the help of this, then we can define a reciprocity symbol
$$\nu_L: K_2(L)\to {\mathbb Z},\qquad (f,g)\mapsto\left|\begin{matrix}\nu_1(f)&\nu_1(g)\\ \nu_2(f)&\nu_2(g)\end{matrix}\right|.$$

With all this, we are now ready to state our reciprocity laws. Let $\pi: X\to\mathrm{Spec}\O_F$ be a regular arithmetic surface defined over a number field $F$ with generic fiber $X_F$. Let $C$ be an irreducible curve on $X$ and $x$ a closed point of $X$. As usual, see e.g., \cite{P2}, 
we obtain an Artinian ring $K_{C,x}$ which is a finite direct sum of two dimensional local fields $L_i$.
Accordingly, we define $$\nu_{C,x}:=\oplus_i[k_i:\F_q]\cdot\nu_{L_i}.$$ Our main theorem of this paper is the following
\vskip 0.20cm
\noindent
{\bf Main Theorem.} ({ Reciprocity Law for Arithmetic Surfaces})

{\it \begin{itemize}
\item [(1)] For a fixed point $x\in X$,
$$\sum_{C:\,x\in C}\nu_{C,x}(f,g)=0,\qquad\forall f,g\in k(X)^*,$$
where $C$ run over all irreducible curves on $X$ which pass through $x$;

\item [(2)] For a fixed vertical prime divisor $V$ on $X$,
$$\sum_{x:\, x\in V}\nu_{V,x}(f,g)=0,\qquad\forall f,g\in k(X)^*,$$
where $x$ run over all closed points of $X$ which lie on $V$;

\item [(3)] For a fixed Horizontal prime divisor $E_P$ on $X$ corresponding an algebraic point $P$ on $X_F$.
For each infinite place $\sigma$ of $F$, denote by $P_{\sigma,j}$ be corresponding closed points on $X_{F_\sigma}$.
Then
$$
\sum_{x:\, x\in E_P}\nu_{E_P,x}(f,g)\cdot \log q_x+\sum_\sigma\sum_jN_\sigma\cdot \nu_{P_{\sigma,j}}(f,g)=0,\qquad\forall f,g\in k(X)^*.
$$
\end{itemize}}

Our proof of this theorem is as follows: for the first two,  similar  as in \cite{O1}, we first construct a $K_2$-central extension of $k(X)^*$ with the help of Kapranov's dimension theory, then we
interpret $\nu_{C,x}$ as a commutator of lifting of elements in some $K_2$ central extension of the group $K(X)^*$, and finally use the splitting of the central extension of adeles.

During this process, we discover that the fundamental reason for our reciprocity law along a vertical curve 
is the Riemann-Roch in dimension one and the intersection in dimension two. With the help of this, with a bit struggle, we finally 
can establish the reciprocity law for horizontal curves with a new construction of arithmetic central extension, arithmetic intersection 
in dimension two and a refined arithmetic Riemann-Roch theorem for arithmetic curves under the frame work of our arithmetic 
adelic theory built up in \cite{SW}.

\section{Reciprocity Laws in Dimension Two}
\subsection{Arithmetic Adelic Complex}

Let $P$ be an algebraic point of $X_F$. Denote by $E_P$ the corresponding prime horizontal divisor of $X$, 
$\cal I_{E_P}$ the ideal sheaf of $E_P\subset X$,
and $\ov E_P$ its arithmetic compactification. Let $\F$ be a coherent sheaf on $X$.
We introduce an arithmetic adelic complex ${\cal A}_{\ov E_P,*}^{\,\,\Ar}(\F)$ by
$$ 
{\cal A}_{\ov E_P,*}^{\,\,\Ar}(\F):=\lim_{\substack{\raw\\ n}}\lim_{\substack{\law\\ m:\, m\geq n}}
{\Bbb A}_{X,*}^{\,\,\Ar}(\F\otimes {\cal I}_{E_P}^n\big/{\cal I}_{E_P}^m),
$$
where ${\Bbb A}_{X,*}^{\,\,\Ar}$ denotes the arithmetic adelic functor introduced in \S 1.2.3 of \cite{SW}.
Since ${\cal A}_{\ov E_P,*}^{\,\,\Ar}(\F)$ is defined over an infinitesimal neighborhood of horizontal curve
$\ov E_P$ in $\ov X$, the complex consists of three terms. For example, when $\cal L$ is an invertible
sheaf, a direct calculation implies that ${\cal A}_{\ov E_P,*}^{\,\,\Ar}(\L)$ is given by
$$
\widehat {k(X)}_{E_P}\times\Big(\prod_{x\in E_P}\big(B_x\otimes_{\widehat{\O}_{X,x}}{\cal L}\big)
\times \big(B_P\otimes_{\widehat{\O}_{X,x}}{\cal L}\,\widehat\otimes_\Q\R\big)\Big)\raw \Bbb A_{\ov E_P}^{\,\,\Ar}
$$
where $B_x:=\O_{E_P,x}((u)),\ B_P=\O_{X_F}\big|_P((u))$ and 
$\Bbb A_{\ov E_P}^{\,\,\Ar}:=\prod_{x\in \ov E_P}'\widehat {k(E_P)}_x((u))$.
That is, its elements are given by ${\bf a}=(a_x)_{x\in\ov E_P}$,  where $a_x=\sum_ia_{i,x}u^i$ 
satisfying that, for any fixed $i$, $(a_{ix})_{x\in \ov E_P}$ is an usual adele of the arithmetic curve   $\ov E_P$.
In particular, ${\cal A}_{\ov E_P,*}^{\,\,\Ar}(\ov \O_X)$ is given by
$$
\widehat {k(X)}_{E_P}\times\big(\prod_{x\in E_P}\O_{E_P,x}((u))
\times \O_{X_F}\big|_P((u))\,\widehat\otimes_\Q\R\big)\raw {\prod}_{x\in \ov E_P}'\widehat {k(E_P)}_x((u)).
$$
Consequently, if we fix a non-zero rational section $s$ of $\L$ such that $s$ does not vanish on $E_P$, 
we may write $\L|_{E_P}$ as 
$\O_{E_P}(\sum_xn_x[x])$ and write $\L_F:=\L|_{X_F}$ as $\L_F=\O_{X_F}(\sum_{Q\in X_F}m_Q[Q])$.
Here $\sum_xn_x[x]=\mathrm{div}(s|_{E_P})$ and $\sum_{Q\in X_F}m_Q[Q]=\mathrm{div}(s|_{X_F})$.
Accordingly, we obtain a complex
$$
\begin{aligned}
\widehat {k(X)}_{E_P}\times\Big(\prod_{x\in E_P}\frak m_{E_P,x}^{-n_x}((u))
\times \prod_{Q\in X_F}\O_{X_F}(Q)^{\otimes m_Q}&\big|_P((u))\,\widehat\otimes_\Q\R\Big)\\
&\raw 
{\prod}_{x\in \ov E_P}'\widehat {k(E_P)}_x((u)).
\end{aligned}
$$
Similarly, for a different non-zero section $s'$ of $\L$ such that $s$ does not vanish on $E_P$, 
we may write  $\L|_{E_P}=\O_{E_P}(\sum_xn_x'[x])$ 
and  $\L_F=\O_{X_F}(\sum_{Q\in X_F}m_Q'[Q])$. Hence, we have the corresponding complex
$$
\begin{aligned}
\widehat {k(X)}_{E_P}\times\Big(\prod_{x\in E_P}\frak m_{E_P,x}^{-n_x'}((u))
\times \prod_{Q\in X_F}\O_{X_F}(Q)^{\otimes m_Q'}&\big|_P((u))\,\widehat\otimes_\Q\R\Big)\\
&\raw {\prod}_{x\in \ov E_P}'\widehat {k(E_P)}_x((u)).
\end{aligned}
$$

\subsection{Numerations in terms of Arakelov intersection}
\subsubsection{Case I}

For a metrized line bundle $\ov\L$ on $X$, set
$$
W_{\ov \L}^{\ar}:=\Big(\prod_{x\in E_P}B_x\otimes_{\widehat{\O}_{X,x}}{\cal L}\Big)
\times \Big(B_P\otimes_{\widehat{\O}_{X,x}}{\ov\L}\,\widehat\otimes_\Q\R\Big).
$$
To count it, for the non-zero rational section $s$ of $\L$ as above, we set
$$
\begin{aligned}
W_{\ov \L,s}^{\ar}:=\Big(\prod_{x\in E_P}\frak m_{E_P,x}^{-n_x}((u))\Big)
\times \Big(\prod_{Q\in X_F}\O_{X_F}(Q)^{\otimes m_Q}&\big|_P((u))\,\widehat\otimes_\Q\R\Big)\\
&\times {\bf e}\Big(\int_{X_\infty}-\log\|s\|d\mu\Big).
\end{aligned}
$$
Similarly, for the section $s'$ as above, we get
$$
\begin{aligned}
W_{\ov \L,s'}^{\ar}:=\Big(\prod_{x\in E_P}\frak m_{E_P,x}^{-n_x'}((u))\Big)
\times \Big(\prod_{Q\in X_F}\O_{X_F}(Q)^{\otimes m_Q'}&\big|_P((u))\,\widehat\otimes_\Q\R\Big)\\
&\times {\bf e}\Big(\int_{X_\infty}-\log\|s'\| d\mu\Big).
\end{aligned}
$$
In particular, for a metrized line bundle $\ov L_1$ satisfying $\L_1\hookrightarrow \L$ with a non-zero
rational section $s_1$ of $\L_1$ which does not vanish along $E_P$, we have
$$
\begin{aligned}
W_{\ov \L,s}^{\ar}\big/W_{\ov \L_1,s_1}^{\ar}=\Big(\prod_{x\in E_P}&\frak m_{E_P,x}^{-n_x+n_{1,x}}((u))\Big)\\
&\times \Big(\prod_{Q\in X_F}\O_{X_F}(Q)^{\otimes m_Q-m_{1Q}}\big|_P((u))\,\widehat\otimes_\Q\R\Big)\\
&\hskip 4.0cm \times {\bf e}\Big(\int_{X_\infty}-\log\frac{\|s\|\ }{\|s_1\|_1} d\mu\Big).
\end{aligned}
$$
To numerate them, we use Arakelov intersection \cite{L}. Let $g(P,Q)$ be the Rrakelov-Green function 
and $G(P,Q)=e^{g(P.Q)}$. Then we introduce the following
\begin{defn}
\begin{itemize}
\item [(1)] $\num_0\big(\frak m_{E_P,x}\big):=-\log q_x;$
\item [(2)] $\num_0\big(\O_{X_F}(Q)|_P\big):=g(Q,P);$
\item [(3)] $\num_0\big(\O_{E_P,x}[[u]]\big/u^n \O_{E_P,x}[[u]]:=n;$
\item [(4)] $\num_0\big(\O_{X_F,P}[[u]]\big/u^m \O_{X_F,P}[[u]]:=m.$
\end{itemize}
\end{defn}
\begin{prop} Let $\ov L_1$, resp. $\ov \L$, be  a motorized line bundle on $X$ satisfying $\L\hookrightarrow \L_1$. And let 
$s_1$ and $s_1'$, resp.  $s$ and $s'$, be  two non-zero rational sections  of $\L_1$, resp. of $\L$, 
which do not vanish  along $E_P$. There exist  
canonical isometries of metrized $\R$-torsors
$$
\begin{aligned}
\Num\big(W_{\ov \L,s}^{\ar}\big)\cong&\ \Num\big(W_{\ov \L,s'}^{\ar}\big),\\[0.2em]
\Num\big(W_{\ov \L_1,s_1}^{\ar}\big/W_{\ov \L,s}^{\ar}\big)\cong&\ \Num\big(W_{\ov \L_1,s_1'}^{\ar}\big/W_{\ov \L,s'}^{\ar}\big).
\end{aligned}
$$
\end{prop}

\begin{proof} With the construction above, our proof for two isometries are similar. 
We here only give the details for the first one.
Clearly, the numerations for the  Laurent series parts of both sides are the same, 
we only need to deal with the coefficient part of both side. Hence, it suffices to show that
{\footnotesize$$
\begin{aligned}
&\num_0\Big(\Big(\prod_{x\in E_P}\frak m_{E_P,x}^{-n_x}\Big)
\times \Big(\prod_{Q\in X_F}\O_{X_F}(Q)^{\otimes m_Q}\big|_P\,\widehat\otimes_\Q\R\Big)
\times {\bf e}\Big(\int_{X_\infty}-\log\|s\|d\mu\Big)\Big) \\
&=\num_0\Big(\Big(\prod_{x\in E_P}\frak m_{E_P,x}^{-n_x'}\Big)
\times \Big(\prod_{Q\in X_F}\O_{X_F}(Q)^{\otimes m_Q'}\big|_P\,\widehat\otimes_\Q\R\Big)
\times {\bf e}\Big(\int_{X_\infty}-\log\|s'\|d\mu\Big)\Big).
\end{aligned}
$$}
By definition, the left hand side is equal to the logarithm of
$$
\prod_{x\in E_P}q_x^{n_x}\prod_{Q\in X_F}G_\infty(P,\di(s_F))\cdot {\bf e}\Big(\int_{X_\infty}-\log\|s\|d\mu\Big),
$$ 
which, by the Arakelov intersection theory \cite{L}, is simply equal to 
$\deg_{\ar}(\ov \L|_{\ov E_P})=c_{1,\ar}(\ov\L)\cdot \ov E_P$.
Similarly, the right hand side is equal to the logarithm of
$$
\prod_{x\in E_P}q_x^{n_x'}\prod_{Q\in X_F}G_\infty(P,\di(s_F'))\cdot {\bf e}\Big(\int_{X_\infty}-\log\|s'\|d\mu\Big),
$$ 
which  is also equal to $\deg_{\ar}(\ov\L|_{\ov E_P})$. 
\end{proof}

\begin{defn} For metrized line bundles $\ov\L_i$ with non-zero rational sections $s_i,\ s_i'$ (i=1,\,2), 
which do not vanish  along $E_P$, we define a metrized $\R$-torsor 
$\big[W_{\ov \L_1,s_1}^{\ar}\big|W_{\ov \L_2,s_2}^{\ar}\big]_2$ by
{\footnotesize$$
\big[W_{\ov \L_1,s_1}^{\ar}\big|W_{\ov \L_2,s_2}^{\ar}\big]_2
:=\lim_{\substack{\law\\ (\ov \L,s)\\ \L\hookrightarrow \L_1, \L_2}}
\Hom_\R\Big(\Num\big(W_{\ov \L_1,s_1}^{\ar}\big/W_{\ov \L,s}^{\ar}\big),
\Num\big(W_{\ov \L_1,s_2}^{\ar}\big/W_{\ov \L,s}^{\ar}\big)\Big).
$$}
\end{defn}

\begin{prop} With the same notation as above, we have a natural isometry of metrized $\R$-torsors
$$
\big[W_{\ov \L_1,s_1}^{\ar}\big|W_{\ov \L_2,s_2}^{\ar}\big]_2
=\big[W_{\ov \L_1,s_1'}^{\ar}\big|W_{\ov \L_2,s_2'}^{\ar}\big]_2.
$$
In particular, the space $\big[W_{\ov \L_1}^{\ar}\big|W_{\ov \L_2}^{\ar}\big]_2$ is well-defined.
\end{prop}

\begin{proof}
This is a direct consequence of the following calculation
$$
\begin{aligned}
\big[W_{\ov \L_1,s_1}^{\ar}&\big|W_{\ov \L_2,s_2}^{\ar}\big]_2\\
=&\lim_{\substack{\law\\ (\ov \L,s)\\ \L\hookrightarrow \L_1, \L_2}}
\Hom_\R\Big(\Num\big(W_{\ov \L_1,s_1}^{\ar}\big/W_{\ov \L,s}^{\ar}\big),
\Num\big(W_{\ov \L_1,s_1}^{\ar}\big/W_{\ov \L,s}^{\ar}\big)\Big)\\
\cong&\lim_{\substack{\law\\ (\ov \L',s')\\ \L'\hookrightarrow \L_1, \L_2}}
\Hom_\R\Big(\Num\big(W_{\ov \L_1,s_1'}^{\ar}\big/W_{\ov \L',s'}^{\ar}\big),
\Num\big(W_{\ov \L_2,s_2'}^{\ar}\big/W_{\ov \L',s'}^{\ar}\big)\Big)\\
&\hskip 2.7cm ~(\text{ by the second isometry of  Prop. 1 above})\\
=&\big[W_{\ov \L_1,s_1'}^{\ar}\big|W_{\ov \L_2,s_2'}^{\ar}\big]_2
\end{aligned}
$$
\end{proof}

\subsubsection{Case II}

Next, more generally, we deal with arbitrary non-zero rational sections $s$ of $\L$, which may vanish along $E_P$.

Let $f_0$ be a non-zero rational function such that $=s\cdot f_s^{-1}$ does not vanish along $E_P$. 
In particular, it makes sense for us to talk about $W^\ar_{\ov\L,s_0}$. 

To construct $W^\ar_{\ov\L,s}$, write $s=s_0\cdot u^{\nu_{E_P}(s)}\cdot f_0$ with $f_0$ a rational function and $s_0$ a section of
 the line bundle $\L_0:=\L(-\nu_{E_P}(s)\cdot E_P)$. There is a natural metric on $\L_0$ obtained as the tensor
 of $\ov \L$ and the $\nu_{E_P}(s)$-th tensor power of $\O_X(\ov E_P)$, namely, $\ov \L_1:=\ov \L(-\nu_{E_P}(s)\ov E_P)$. 
Our idea of constructing $W^\ar_{\ov\L,s}$ is to use the \lq existence' of a natural  decomposition 
$$
W^\ar_{\ov\L,s}=W^\ar_{\ov\L_0,s_0}\otimes W_{\O_X(\ov E_P),u}^{\otimes \nu_{E_P}(s)}\otimes W_{\ov O_X,f_0}.
$$
Since $s_0$ and $f_0$ do not vanish along $E_P$,  the spaces
$W^\ar_{\ov\L_0,s_0}$ and  $W_{\ov O_X,f_0}$ make sense. Therefore, to construct $W^\ar_{\ov\L,s}$ 
for a general non-zero section $s$, it suffices to define $W_{\O_X(\ov E_P),u}$ which we write simply as $W_{\ov E_P,u}$.
\vskip 0.20cm
With the discussion of Case I above in mind, in particular, the role of Arakelov intersection,
to construct $W_{\ov E_P,u}$, we first recall the Arakelov
adjunction formula. Let $\K_\pi$ be the canonical divisor of arithmetic surface 
$\pi:X\to Y:=\mathrm{Spac}\,\O_F$. Set $E_P=\mathrm{Spec} A$.
Then, see e.g., Cor. 5.5 at p.99 of \cite{L}, we have, globally, 
$$
\ov\K_\pi\cdot \ov E_P+\ov E_P^2={\bf d}_{E_P/Y}+{\bf d}_{\lambda}(\ov E_P).
$$
Here 

\begin{itemize}
\item [(a)] ${\bf d}_{\lambda}(\ov E_P)=\sum_{\sigma\in S_\infty}N_\sigma{\bf d}_{g,\sigma}(\ov E_P)$
with ${\bf d}_{g,\sigma}(\ov E_P):=\sum_{i\leq j}g(P_i,P_j)$ where $\{P_i\}$ are the collection of 
conjugating points on the Riemann surface $X_{\sigma}$ corresponding to the algebraic 
point $P$ of $X_F$; and
\item[(b)] ${\bf d}_{E_P/Y}=-\log\, (W_{E_P/Y}:\O_F)$, where $W_{E_P/R}$ is the 
dualizing module of $E_P$ over $Y$, defined by the fractional ideal of $F(P)$
$$
W_{E_P/Y}:=\{b\in F(P): \mathrm{Tr}(bA)\subset \O_F\}.
$$
For later use, we write $W_{E_P/Y}:=\prod_{x\in E_P}\frak m_{E_P,x}^{b_x}$.
\end{itemize}

Moreover, by reside theorem, i.e., Theorem 4.1 of Chapter IV in \cite{L}, we know that
the natural residue map induces a canonical isomorphism
$$
\mathrm{res}: \K_\pi( E_P)|_{E_P}\simeq W_{E_P/Y}^{\widetilde ~},
$$
where $W_{E_P/Y}^{\widetilde ~}$ denotes the sheaf on affine curve $E_P$ associated
to the $A$-module $W_{E_P/Y}$.

Now we are ready o define $W_{\ov E_P, u}^\ar$. As in the definition of $W_{\L,s}$ in Case I above,
$W_{\ov E_P, u}^\ar$ admits a natural decomposition:
$$
W_{\ov E_P,u}^\ar=W_{E_P,u}^{\mathrm{fin}}\times W_{\ov E_P,u}^\infty\cdot {\bf e}\Big(c_{\ov E_P}\Big).
$$
It is too naive to set
$$
W_{E_P,u}^{\mathrm{fin}}=\prod_{x\in E_P}\,u\cdot \O_{E_P,x}((u))
\times\prod_{Q\in X_F}u^{\nu_P(f)}\O_{X_F}|_P((u)).
$$
A bit thought in terms of Arakelov intersection, in particular, the adjunction formula and the associated residue theorem,
leads to the use of the formula
$$
W_{\ov E_P,u}^\ar\otimes W_{\ov \K_\pi, \omega_0}=W_{\ov K_\pi(\ov E_P)}^\ar.
$$
Here $\omega_0$ is a certain non-zero rational section of $\K_\pi$ 
which does not vanish along $E_P$.
Since, by the residue theorem recalled above, $\K_\pi(E_P)|_{E_P}=W_{E_P/Y}^{\widetilde{~}},$
accordingly, we define
$$
\begin{aligned}
W_{E_P,u}^{\mathrm{fin}}:=&\prod_{x\in E_P}u\cdot \frak m_{E_P,x}^{b_x-\nu_x(\omega_0|_{E_P})}((u)),\\
W_{E_P,u}^\infty:=&\prod_{Q\in X_F} u\cdot \O_{X_F} \big(Q\big)^{\otimes -\nu_Q(\omega_0|_{X_F})}\big|_P((u))\widehat\otimes_\Q \R,\\
c_{\ov E_P}:=&\int_{X_\infty}\log\|\omega_0\|d\mu-{\bf d}_\lambda(\ov E_P).
\end{aligned}
$$
Such defined, using the adjunction formula and the residue theorem again, we have
$$
\begin{aligned}
\num_0&\Big(\prod_{x\in E_P} \frak m_{E_P,x}^{b_x-\nu_x(\omega_0|_{E_P})}
\times \prod_{Q\in X_F} \cdot \O_{X_F} \big(Q\big)^{\otimes -\nu_Q(\omega_0|_{X_F})}\big|_P\widehat\otimes_\Q \R\\
&\hskip 5.0cm \cdot{\bf e}\big(\int_{X_\infty}\log\|\omega_0\|d\mu-{\bf d}_\lambda(\ov E_P)\big)\Big)\\
=&-{\bf d}_{E_P/Y}+\ov \K_\pi\cdot \ov E_P-{\bf d}_\lambda(\ov E_P)
=\ov \K_\pi\cdot \ov E_P-\big(\ov \K_\pi+\ov E_P\big)\cdot \ov E_P\\
=&\ov E_P \cdot \ov E_P.
\end{aligned}
$$
All in all, we are ready to introduce the following

\begin{defn} Let $\ov \L$ be a motorized line bundle on $X$ with $s$ a non-zero rational section. With respect to 
a decompose $s=s_0\cdot u^{\nu_{E_P}(s)}\cdot f_0$, we define  
$$
W_{\ov\L,s}^\ar:=\big(W_{\ov E_P, u}^\ar\big)^{\otimes \nu_P(f_s)}\otimes 
W_{\ov\L\otimes \ov E_P^{\otimes -\nu_P(f_s)},f_0s_0}^\ar.
$$
\end{defn}

Then the numeration of the coefficient part is equal to 
$$
\Big(\nu_P(f_s)\,\ov E_P+c_{1,\ar}\big(\ov\L\otimes \ov E_P^{\otimes -\nu_P(f_s)}\big)+c_{1,\ar}(\O_X)\Big)\cdot \ov E_P=
c_{1,\ar}(\ov\L)\cdot \ov E_P.
$$
Hence we have proved the following

\begin{thm} For any non-zero rational sections $s$ and $s'$ of $\L$, we have
$$
\Num\big(W_{\ov \L,s}^{\ar}\big)\cong \Num\big(W_{\ov \L,s'}^{\ar}\big).
$$
In particular, we may write $\Num\big(W_{\ov \L,s}^{\ar}\big)$ simply as 
$\Num\big(W_{\ov \L}^{\ar}\big).$
\end{thm} 

Consequently, we may introduce the following 

\begin{defn} For metrized line bundles $\ov\L_i$ with non-zero rational sections $s_i,\ s_i'$ of $\L_i$
(i=1,\,2),  we define a metrized $\R$-torsor $\big[W_{\ov \L_1,s_1}^{\ar}\big|W_{\ov \L_2,s_2}^{\ar}\big]_2$ by
$$
\big[W_{\ov \L_1,s_1}^{\ar}\big|W_{\ov \L_2,s_2}^{\ar}\big]_2
:=\lim_{\substack{\law\\ (\ov \L,s)\\ \L\hookrightarrow \L_1, \L_2}}
\Hom_\R\Big(\Num\big(W_{\ov \L_1,s_1}^{\ar}\big/W_{\ov \L,s}^{\ar}\big),
\Num\big(W_{\ov \L_1,s_2}^{\ar}\big/W_{\ov \L,s}^{\ar}\big)\Big).
$$
\end{defn}

With the same proof as that for Proposition 2 in \S2.2.1, we have proved the following

\begin{thm} With the same notation as above, we have a natural isometry of metrized $\R$-torsors
$$
\big[W_{\ov \L_1,s_1}^{\ar}\big|W_{\ov \L_2,s_2}^{\ar}\big]_2
=\big[W_{\ov \L_1,s_1'}^{\ar}\big|W_{\ov \L_2,s_2'}^{\ar}\big]_2.
$$
In particular, the space $\big[W_{\ov \L_1}^{\ar}\big|W_{\ov \L_2}^{\ar}\big]_2$ is well-defined.
\end{thm}

\subsection{Numerations in terms of arithmetic adelic cohomology}
\subsubsection{Case I}
With the numeration in terms of intersection developed, next we introduce a numeration
in terms of one dimensional cohomology theory. As we will see later reciprocity laws for arithmetic
surfaces then can be proved using our refined arithmetic Riemann-Roch theorem \cite{SW} or better \cite{W}.

\begin{prop} Let $\ov \L$ be a metrized invertible sheaf on $X$ and $s$ is a non-zero rational section of $\L$. 
Assume that $s$ does not vanish along $E_P$, then the cohomology groups of the complex 
${\cal A}_{\ov E_P,*}^{\,\,\Ar}(\ov\L, s)$ is given by
$$
\begin{aligned}
H^0\big({\cal A}_{\ov E_P,*}^{\,\,\Ar}(\ov\L,s)\big)=&H_\ar^0\big(\ov E_P,\ov\L|_{\ov E_P}\big)((u)),\\
&\\
H^1\big({\cal A}_{\ov E_P,*}^{\,\,\Ar}(\ov\L,s)\big)=&H_\ar^1\big(\ov E_P,\ov\L|_{\ov E_P}\big)((u)).
\end{aligned}
$$
\end{prop}
\begin{proof} This is a direct consequence of arithmetic adelic cohomology theory developed in \cite{SW}.
Indeed, for the invertible sheaf $\L$,  by the definition of the complex ${\cal A}_{\ov E_P,*}^{\,\,\Ar}(\ov\L,s)$, its
cohomology is given by
$$
\lim_{\substack{\raw\\ n}}\lim_{\substack{\law\\ m:\, m\geq n}}
H_\ar^i(\ov E_P,\L\otimes {\cal I}_{E_P}^n\big/{\cal I}_{E_P}^m),\qquad i=0,1.
$$
This then already implies 
$H^0\big({\cal A}_{\ov E_P,*}^{\,\,\Ar}(\L)\big)=H_\ar^0\big(\ov E_P,\L|_{\ov E_P}\big)((u))$.
To prove $H^1\big({\cal A}_{\ov E_P,*}^{\,\,\Ar}(\L)\big)=H_\ar^1\big(\ov E_P,\L|_{\ov E_P}\big)((u)),$ 
we use the fact that,  for horizontal curve $\ov E_P$, the maps 
$$
\varphi_{l,n;m}: H_\ar^1(\ov E_P,\L\otimes {\cal I}_{E_P}^n\big/{\cal I}_{E_P}^m)
\raw 
H_\ar^1(\ov E_P,\L\otimes {\cal I}_{E_P}^l\big/{\cal I}_{E_P}^m)
$$ 
are all injective, since
Parshin-Beilinson's $H^0(E_P,\cdot)$ is exact. (In fact, since $E_P$ is affine,
Parshin-Beilinson's $H^1(E_P,\cdot)$ is always 0.)
\end{proof}

Accordingly, we obtain two metrized $\R$-torsors
$$
\Num\big(H^0\big({\cal A}_{\ov E_P,*}^{\,\,\Ar}(\ov \L,s)\big)\big)\qqan 
\Num\big(H^1\big({\cal A}_{\ov E_P,*}^{\,\,\Ar}(\ov \L,s)\big)\big).
$$

\begin{defn} For metrized line bundle $\ov \L$ on $X$ and a non-zero rational section $s$ of $\L$
which does not vanish along $E_P$, we define a metrized $\R$-torsor 
$\Num\big({\cal A}_{\ov E_P,*}^{\,\,\Ar}(\ov \L,s)\big)$ by setting
$$
\begin{aligned}
\Num\big(&{\cal A}_{\ov E_P,*}^{\,\,\Ar}(\ov \L,s)\big)\\
:=&\Hom_\R\Big(\Num\big(H^1\big({\cal A}_{\ov E_P,*}^{\,\,\Ar}(\ov \L,s)\big)\big), 
\Num\big(H^0\big({\cal A}_{\ov E_P,*}^{\,\,\Ar}(\ov \L,s)\big)\big)\Big).
\end{aligned}
$$
\end{defn}

To go further, we need the following

\begin{defn} We introduce a canonical numeration $\num_0$ by
\begin{itemize}
\item [(1)] $\num_0\big(H_\ar^0\big(\ov E_P,\ov\L|_{\ov E_P}\big)\big):=h_\ar^0\big(\ov E_P,\ov \L|_{\ov E_P}\big);$
\item [(2)] $\num_0\big(H_\ar^1\big(\ov E_P,\ov \L|_{\ov E_P}\big)\big):=h_\ar^1\big(\ov E_P,\ov \L|_{\ov E_P}\big);$
\item [(3)] $\num_0\big(A[[u]]\big/u^n A[[u]]:=n\cdot \num_0(A)$ for any numerable locally compact space $A$
\end{itemize}
\end{defn}

With all this, we are now ready to prove the following

\begin{thm} Let  $\ov \L_i$ be metrized line bundle on $X$ and $s_i$ be a non-zero rational section of $\L_i$
which does not vanish along $E_P$\ (i=1,2). 
There is a canonical isometry of metrized $\R$-torsors
$$
\big[W_{\ov \L_1,s_1}^{\ar}\big|W_{\ov \L_2,s_2}^{\ar}\big]_2\cong
\Hom_\R\Big(\Num\big({\cal A}_{\ov E_P,*}^{\,\,\Ar}(\ov \L_1,s_1)\big),
\Num\big({\cal A}_{\ov E_P,*}^{\,\,\Ar}(\ov \L_2,s_2)\big)\Big).
$$
\end{thm}

\begin{proof}
It suffices to identify special numerations in both sides canonically. 
Note that, for them, the Laurent series parts work in the same way, it suffices to treat their coefficient parts.
From Prop ???, we know that $\num_0$ for the left hand side gives
$$
\deg_{\ar}\big((\ov\L_1-\ov\L_2)|_{\ov E_P}\big).\eqno(*)
$$ 
To complete the proof, it suffices to notice that
$\num_0$ for the coefficient part of $\Num\big({\cal A}_{\ov E_P,*}^{\,\,\Ar}(\ov \L_1,s_1)\big)$ equals
$$
\begin{aligned}
&h_\ar^1\big(\ov E_P,\ov \L_1|_{\ov E_P}\big)-h_\ar^0\big(\ov E_P,\L_1|_{\ov E_P}\big)\\
=&-\chi_\ar(\ov E_P, \ov \L_1|_{\ov E_P})
=-\deg_{\ar}\big((\ov\L_1)|_{\ov E_P}\big)-\frac{1}{2}\log|\Delta_{k(P)}|.
\end{aligned}
$$ 
Similarly, 
$\num_0$ for the coefficient part of $\Num\big({\cal A}_{\ov E_P,*}^{\,\,\Ar}(\ov \L_2,s_2)\big)$ equals
$$
\begin{aligned}
&h_\ar^1\big(\ov E_P,\ov \L_2|_{\ov E_P}\big)-h_\ar^0\big(\ov E_P,\L_2|_{\ov E_P}\big)\\
=&-\chi_\ar(\ov E_P, \ov \L_2|_{\ov E_P})
=-\deg_{\ar}\big((\ov\L_2)|_{\ov E_P}\big)-\frac{1}{2}\log|\Delta_{k(P)}|.
\end{aligned}
$$ 
Therefore, the $\num_0$ for the coefficient part of the right hand gives
$$
\begin{aligned}
-\Big(-\deg_{\ar}\big(&(\ov\L_1)|_{\ov E_P}\big)-\frac{1}{2}\log|\Delta_{k(P)}|\Big)
+\Big(-\deg_{\ar}\big((\ov\L_2)|_{\ov E_P}\big)-\frac{1}{2}\log|\Delta_{k(P)}|\Big)\\[0.8em]
=&\deg_{\ar}\big((\ov\L_1-\ov\L_2)|_{\ov E_P}\big).
\end{aligned}\eqno(**)
$$ 
This coincides with the canonical numeration $\num_0$ for the coefficient of the space
$\big[W_{\ov \L_1,s_1}^{\ar}\big|W_{\ov \L_2,s_2}^{\ar}\big]_2$ obtained using Arakelov intersection
given in $(*)$.
\end{proof}

\subsubsection{Case II}
With above discussion on special rational sections, next we treat arbitrary non-zero rational sections.
Let then $s$ be a non-zero rational section of $\L$ and write $s=s_0u^{\nu_P(f_s)}f_0$ with 
$s_0$ a rational section of $\L(-\nu_P(f_s)E_P)$, $f_0$ a rational function on $X$ which does not vanish
along $E_P$. Then, 
by definition, 
$$
W_{\ov\L,s}^\ar=\big(W_{\ov E_P, u}^\ar\big)^{\otimes \nu_P(f_s)}\otimes 
W_{\ov\L\otimes \ov E_P^{\otimes -\nu_P(f_s)},s_0}^\ar
\otimes W_{\ov\O_X, f_0}^\ar.
$$
And to calculate $H^i_\ar\big({\cal A}_{\ov E_P,*}^{\,\,\Ar}(\ov \L,s)\big)\big), i=0,1$, we use the complex
$$
k(X)^*\times W_{\ov\L,s}^\ar\buildrel \varphi\over\raw \Bbb A_{\ov E_P}^{\,\,\Ar}.
$$
Here, if we write 
$\mathrm{div}(s_0f_0|_{E_P})=\sum_xn_{0,x}[x],\  \mathrm{div}(s_0f_0|_{X_F})=\sum_Qm_{0,Q}[Q]$,
by definition,
$$
\begin{aligned}
W_{\ov\L,s}^\ar=&\prod_{x\in E_P} u^{\nu_P(s)}\cdot 
\frak m_{E_P,x}^{\nu_P(s)(b_x-\nu_x(\omega_0|_{E_P}))-n_{0,x}}((u))\\
&\times\prod_{Q\in X_F}u^{\nu_P(s)}\cdot 
\O_{X_F}\big(Q\big)^{-\nu_P(s)\nu_Q(\omega_0|_{X_F})+m_{0,Q}}\big|_P((u))\\
&\cdot{\bf e}\Big(\int_{X_\infty}\big(\nu_P(s)\log\|\omega_0\|-\log\|s_0f_0\| \big)d\mu
+\nu_P(s){\bf d}_{\lambda}(\ov E_P)\Big).
\end{aligned}
$$
Consequently, $H^0_\ar\big({\cal A}_{\ov E_P,*}^{\,\,\Ar}(\ov \L,s)\big)\big)=\mathrm{Ker}(\varphi)$ 
which can be described as
$$\begin{aligned}
H^0_\ar\big({\cal A}_{\ov E_P,*}^{\,\,\Ar}&(\ov \L,s)\big)\big)\\
=& u^{\nu_P(s)}\prod_{x\in E_P}  \frak m_{E_P,x}^{\nu_P(s)(b_x-\nu_x(\omega_0|_{E_P}))-n_{0,x}}\\
&\times\prod_{Q\in X_F} G\big(Q,P\big)^{-\nu_P(s)\nu_Q(\omega_0|_{X_F})+m_{0,Q}}\\
&\cdot{\bf e}\Big(\int_{X_\infty}\big(\nu_P(s)\log\|\omega_0\|-\log\|s_0f_0\| \big)d\mu
+\nu_P(s){\bf d}_{\lambda}(\ov E_P)\Big)((u)),
\end{aligned}
$$
and the corresponding quotient space 
$H^1_\ar\big({\cal A}_{\ov E_P,*}^{\,\,\Ar}(\ov \L,s)\big)\big)=\mathrm{Coker}(\varphi)$,
whose explicit description we leave to the reader. (See e.g. \S1.2.4 of \cite{SW}.)

Since $H^i_\ar\big({\cal A}_{\ov E_P,*}^{\,\,\Ar}(\ov \L,s)\big)\big),\ i=0,1$ are ind-pro topology spaces
induced from Laurent series with coefficients numerable locally compact spaces,
consequently, we obtain two metrized $\R$-torsors
$$
\Num\big(H^0\big({\cal A}_{\ov E_P,*}^{\,\,\Ar}(\ov \L,s)\big)\big)\qqan 
\Num\big(H^1\big({\cal A}_{\ov E_P,*}^{\,\,\Ar}(\ov \L,s)\big)\big).
$$
As it stands, unlike in Def. 5, it is not easy to describe the numerations 
for the coefficient parts of them separately under $\num_0$. However,
their difference can be treated well. For this we introduce the following

\begin{defn} For metrized line bundle $\ov \L$ on $X$ and a non-zero rational section $s$ of $\L$, 
we define a metrized $\R$-torsor 
$\Num\big({\cal A}_{\ov E_P,*}^{\,\,\Ar}(\ov \L,s)\big)$ by setting
$$
\begin{aligned}
\Num\big(&{\cal A}_{\ov E_P,*}^{\,\,\Ar}(\ov \L,s)\big)\\
:=&\Hom_\R\Big(\Num\big(H^1\big({\cal A}_{\ov E_P,*}^{\,\,\Ar}(\ov \L,s)\big)\big), 
\Num\big(H^0\big({\cal A}_{\ov E_P,*}^{\,\,\Ar}(\ov \L,s)\big)\big)\Big).
\end{aligned}
$$
\end{defn}

\begin{lem} With the same notation as above, the canonical numeration $\num_0$ 
for the coefficient part of $\Num\big({\cal A}_{\ov E_P,*}^{\,\,\Ar}(\ov \L,s)\big)$
is simply $-\chi_\ar\big(\ov E_P, \ov\L|_{\ov E_P}).$
\end{lem}

\begin{proof} Indeed, by definition, the quantity we seek is equal to the negative of
$$\begin{aligned}\chi_\ar\Big(\ov E_P,& \prod_{x\in E_P}  \frak m_{E_P,x}^{\nu_P(s)(b_x-\nu_x(\omega_0|_{E_P}))-n_{0,x}}\\
&\times\prod_{Q\in X_F} G\big(Q,P\big)^{-\nu_P(s)\nu_Q(\omega_0|_{X_F})+m_{0,Q}}\\
&\cdot{\bf e}\Big(\int_{X_\infty}\big(\nu_P(s)\log\|\omega_0\|-\log\|s_0f_0\| \big)d\mu+\nu_P(s){\bf d}_{\lambda}(\ov E_P)
\Big)\Big).
\end{aligned}
$$
That is,
$$
\begin{aligned}
-&\Big(\nu_P(s){\bf d}_{E_P/Y}-\nu_P(s)\cdot c_{1,\ar}(\ov\K_\pi)\cdot \ov E_P+\nu_P(s){\bf d}_\lambda{\ov E_P}\\
&\hskip 4.50cm +\big(c_{1,\ar}(\ov\L)-\nu_P\ov E_P\big)\cdot \ov E_P\Big)+\frac{1}{2}\log|\Delta_{k(P)}|\\
&\hskip 4.9cm(\text{by the residue theorem and definition})\\
=&-\Big(\nu_P(s)\big(c_{1,\ar}(\ov\K_\pi)+\ov E_P\big)\cdot \ov E_P-\nu_P(s)\cdot c_{1,\ar}(\ov\K_\pi)\cdot \ov E_P\\
&\hskip 4.50cm +\big(c_{1,\ar}(\ov\L)-\nu_P\ov E_P\big)\cdot \ov E_P\Big)+\frac{1}{2}\log|\Delta_{k(P)}|\\
&\hskip 6.9cm(\text{by the adjunction formula})\\
=&-c_{1,\ar}(\ov\L)\cdot \ov E_P+\frac{1}{2}\log|\Delta_{k(P)}|\\
=&-\chi_\ar\big(\ov E_P, \ov\L|_{\ov E_P})\\
&\hskip 6.2cm(\text{by the Riemann-Roch formula}).\\
\end{aligned}
$$

\end{proof}

With all this, we are now ready to prove the following

\begin{thm} Let  $\ov \L_i$ be metrized line bundle on $X$ and $s_i$ be a non-zero rational section of $\L_i$
\ (i=1,2). There is a canonical isometry of metrized $\R$-torsors
$$
\big[W_{\ov \L_1,s_1}^{\ar}\big|W_{\ov \L_2,s_2}^{\ar}\big]_2\cong
\Hom_\R\Big(\Num\big({\cal A}_{\ov E_P,*}^{\,\,\Ar}(\ov \L_1,s_1)\big),
\Num\big({\cal A}_{\ov E_P,*}^{\,\,\Ar}(\ov \L_2,s_2)\big)\Big).
$$
\end{thm}

\begin{proof}
It suffices to identify special numerations in both sides canonically. 
Note that, for both sides, the Laurent series parts work in the same way, it suffices to treat their coefficient parts.
From our proof of Prop 3. and Thm. 4, we know that $\num_0$ for the left hand side gives
$$
\deg_{\ar}\big((\ov\L_1-\ov\L_2)|_{\ov E_P}\big).\eqno(*)
$$ 
To complete the proof, it suffices to notice that
$\num_0$ for the coefficient part of $\Num\big({\cal A}_{\ov E_P,*}^{\,\,\Ar}(\ov \L_1,s_1)\big)$ gives
$-\chi_\ar(\ov E_P, \ov \L_1|_{\ov E_P}).$ 
Indeed, similarly, 
$\num_0$ for the coefficient part of $\Num\big({\cal A}_{\ov E_P,*}^{\,\,\Ar}(\ov \L_2,s_2)\big)$ gives
$-\chi_\ar(\ov E_P, \ov \L_2|_{\ov E_P}).$ 
Therefore, the $\num_0$ for the coefficient part of the right hand gives
$$
\begin{aligned}
-\Big(-\deg_{\ar}\big(&(\ov\L_1)|_{\ov E_P}\big)+\frac{1}{2}\log|\Delta_{k(P)}|\Big)
+\Big(-\deg_{\ar}\big((\ov\L_2)|_{\ov E_P}\big)+\frac{1}{2}\log|\Delta_{k(P)}|\Big)\\[0.8em]
=&\deg_{\ar}\big((\ov\L_1-\ov\L_2)|_{\ov E_P}\big).
\end{aligned}\eqno(**)
$$ 
This coincides with the canonical numeration $\num_0$ for the coefficient of the space
$\big[W_{\ov \L_1,s_1}^{\ar}\big|W_{\ov \L_2,s_2}^{\ar}\big]_2$ obtained using Arakelov intersection
given in $(*)$.
\end{proof}

\subsection{Arithmetic Central Extension $k(X)^*_{W_{\ov \O_X}^{\ar}}$ }

To define arithmetic central extension $k(X)^*_{W_{\ov \O_X}^\ar}$, we first consider the 
action of $f\in k(X)^*$ on $W_{\ov \O_X}^\ar$. Assume for the time being that
$f$ does not vanish along $E_P$. Set 
$\mathrm{div}(f|_{E_P})=\sum_xn_x[x], \ \mathrm{div}(f|_{X_F})=\sum_Qm_Q[Q]$.
In terms of pure algebraic structures involved,
$$
W_{ \O_X}^\ar=\prod_{x\in E_P}\O_{E_P,x}((u))\times\prod_{Q\in X_F}\O_{X_F,Q}|_P((u)).
$$
Thus, algebraically, 
$$
f\cdot W_{ \O_X}^\ar=\prod_{x\in E_P}\frak m_{E_P,x}^{-n_x}((u))
\times\prod_{Q\in X_F}\O_{X_F}(Q)^{m_Q}|_P((u)).
$$ 
Accordingly, arithmetically, we define the action of $f$ on $W_{\ov \O_X}^\ar$ by
$$
\begin{aligned}
f W_{\ov \O_X}^\ar:=&\prod_{x\in E_P}\frak m_{E_P,x}^{-n_x}((u))
\times\prod_{Q\in X_F}\O_{X_F}(Q)^{m_Q}|_P((u)) \cdot{\bf e}\Big(\int_{X_\infty}-\log\|f\|d\mu\Big)\\
=&W_{\ov \O_X(\mathrm{div}_{\ar}(f)),\,f}^\ar\ .
\end{aligned}
$$
Here $\mathrm{div}_{\ar}(f)$ denotes the Arakelov divisor associated to the rational function $f$,
and we use  
$\ov \O_X(\mathrm{div}_{\ar}(f))$ to denote the metrized line bundle associated to the Arakelov divisor
$\mathrm{div}_{\ar}(f)$.

Motivated by the above discussion, for an arbitrary non-zero rational function $f$, let the action of $f$ 
on $W_{\ov \O_X,1}^\ar$ by 
$$
f\cdot W_{\ov \O_X,1}^\ar=W_{\ov \O_X(\mathrm{div}_{\ar}(f)),\,f}^\ar\ .
$$
Then we have the following

\begin{prop} For $f,\,g\in k(X)^*$, we have natural isometry
$$
[W_{\ov \O_X}^\ar:fW_{\ov \O_X}^\ar]_2\otimes [fW_{\ov \O_X}^\ar:fgW_{\ov \O_X}^\ar]_2
\cong [W_{\ov \O_X}^\ar:fgW_{\ov \O_X}^\ar]_2.
$$
\end{prop}

\begin{proof} Essentially, this is because arithmetic intersection works well.
Indeed, by our definition, 
$\dis{f W_{\ov \O_X}^\ar=W_{\ov \O_X(\mathrm{div}_{\ar}(f))}^\ar}.$ 
Hence, it suffices to prove that there is an isometry
{\footnotesize
$$[W_{\ov \O_X}^\ar:W_{\ov \O_X(\mathrm{div}_{\ar}(f))}^\ar]_2
\otimes [W_{\ov \O_X(\mathrm{div}_{\ar}(f))}^\ar:W_{\ov \O_X(\mathrm{div}_{\ar}(fg))}^\ar]_2
\cong [W_{\ov \O_X}^\ar:W_{\ov \O_X(\mathrm{div}_{\ar}(fg))}^\ar]_2.
$$\\[-3em]}

\noindent
But, similarly, as in the proof of Prop. 2, this is clear by the definition of $\num_0$ in terms of intersections, 
due to the following trivial cancelation
$$
\mathrm{div}_{\ar}(f)+\Big(-(\mathrm{div}_{\ar}(f))+(\mathrm{div}_{\ar}(fg))\Big)
=(\mathrm{div}_{\ar}(fg)).
$$
That is, 
\footnotesize{$$
\begin{aligned}
&[W_{\ov \O_X}^\ar:W_{\ov \O_X(\mathrm{div}_{\ar}(f))}^\ar]_2
\otimes [W_{\ov \O_X(\mathrm{div}_{\ar}(f))}^\ar:W_{\ov \O_X(\mathrm{div}_{\ar}(fg))}^\ar]_2\\[0.2em]
=&\lim_{\substack{\law\\ (\ov \L,s),\ \L\hookrightarrow \O_X\\ \L\hookrightarrow \O_X(\mathrm{div}_{\ar}(f))}}
\Hom_\R\Big(\Num\big(W_{\ov \O_X,1}^{\ar}\big/W_{\ov \L,s}^{\ar}\big),
\Num\big(W_{\ov \O_X(\mathrm{div}_{\ar}(f)),f}^{\ar}\big/W_{\ov \L,s}^{\ar}\big)\Big)\\
\otimes&
\lim_{\substack{\law\\ (\ov \L,s)\ \L\hookrightarrow \O_X(\mathrm{div}_{\ar}(f))\\  \L\hookrightarrow\O_X(\mathrm{div}_{\ar}(fg))}}
\Hom_\R\Big(\Num\big(W_{\ov \O_X(\mathrm{div}_{\ar}(f)),f}^{\ar}\big/W_{\ov \L,s}^{\ar}\big),
\Num\big(W_{\ov \O_X(\mathrm{div}_{\ar}(fg)),fg}^{\ar}\big/W_{\ov \L,s}^{\ar}\big)\Big)\\[0.2em]
\cong&\lim_{\substack{\law\\ (\ov \L,s)\ \L\hookrightarrow \O_X, \O_X(\mathrm{div}_{\ar}(f))\\  \L\hookrightarrow \O_X(\mathrm{div}_{\ar}(fg))}}
\Big(\Hom_\R\Big(\Num\big(W_{\ov \O_X,1}^{\ar}\big/W_{\ov \L,s}^{\ar}\big),
\Num\big(W_{\ov \O_X(\mathrm{div}_{\ar}(f)),f}^{\ar}\big/W_{\ov \L,s}^{\ar}\big)\Big)\\
&\hskip 2.0cm\otimes
\Hom_\R\Big(\Num\big(W_{\ov \O_X(\mathrm{div}_{\ar}(f)),f}^{\ar}\big/W_{\ov \L,s}^{\ar}\big),
\Num\big(W_{\ov \O_X(\mathrm{div}_{\ar}(fg)),fg}^{\ar}\big/W_{\ov \L,s}^{\ar}\big)\Big)\Big)\\[0.2em]
\cong&\lim_{\substack{\law\\ (\ov \L,s)\\ \L\hookrightarrow \O_X,  \O_X(\mathrm{div}_{\ar}(fg))}}
\Hom_\R\Big(\Num\big(W_{\ov \O_X,1}^{\ar}\big/W_{\ov \L,s}^{\ar}\big),
\Num\big(W_{\ov \O_X(\mathrm{div}_{\ar}(fg)),fg}^{\ar}\big/W_{\ov \L,s}^{\ar}\big)\Big)\\[0.2em]
=& [W_{\ov \O_X}^\ar:W_{\ov \O_X(\mathrm{div}_{\ar}(fg))}^\ar]_2.
\end{aligned}
$$}
\end{proof}
Consequently, we may make the following

\begin{defn} We define an arithmetic central extension  group $k(X)^*_{W_{\ov \O_X}^\ar}$ by the following data:
\begin{itemize}
\item [(a)] As a set, its elements are given by pairs $(f,\alpha)$ with $f\in k(X)^*, \alpha\in [W_{\ov \O_X}^\ar:fW_{\ov \O_X}^\ar]_2$;
\item [(b)] For the group law, its multiplication is defined by $$(f,\alpha)\circ(g,\beta):=(fg,\alpha\circ f(\beta),$$
where $$\alpha\circ f(\beta):=\alpha\otimes f(\beta).$$
\end{itemize}
\end{defn}

Indeed, since $\alpha\in [W_{\ov \O_X}^\ar:fW_{\ov \O_X}^\ar]_2,\ \beta\in [W_{\ov \O_X}^\ar:gW_{\ov \O_X}^\ar]_2$, we have
$$
f(\beta)\in [fW_{\ov \O_X}^\ar:fgW_{\ov \O_X}^\ar]_2,\qqan\alpha\circ f(\beta)\in  [W_{\ov \O_X}^\ar:fgW_{\ov \O_X}^\ar]_2.
$$
So (b) is well-defined.\\

\subsection{Splitness}

Concerning the group $k(X)^*_{W_{\ov \O_X}^\ar}$, we have the following

\begin{prop} We have 
\begin{itemize}
\item [(1)] The group $k(X)^*_{W_{\ov \O_X}^\ar}$ is a central extension of $k(X)^*$ by $\R$. In particular, we have
an exact sequence
$$
0\to \R\raw k(X)^*_{W_{\ov \O_X}^\ar}\raw k(X)^*\to 1.
$$
\item [(2)] The short exact sequence (1) splits.
\end{itemize}
\end{prop}

\begin{proof}
(1) follows directly from the definition.  As for (2), we introduce one more central extension 
${k(X)^{*'}_{W_{\ov \O_X}^\ar}}$ of the group $k(X)^*$ by $\R$ as follows:
\begin{itemize}
\item [(a)] As a set, its elements are given by pairs $(f,\alpha')$, where $f\in k(X)^*$ and 
$\alpha'\in\Hom_\R\big(\Num({\cal A}_{\ov E_P,*}^{\,\,\Ar}(\ov \O_X), \Num({\cal A}_{\ov E_P,*}^{\,\,\Ar}(f\ov \O_X))\big)$;
\item[(b)] For the group law, its multiplication is defined by 
$$
(f,\alpha')\circ(g,\beta'):=(fg,\alpha'\circ f(\beta').
$$
\end{itemize}
Now it is sufficient to use  Theorem 8. Indeed, by Theorem 8, we obtain a natural isomorphism
$k(X)^*_{W_{\ov \O_X}^\ar}\simeq {k(X)^{*'}_{W_{\ov \O_X}^\ar}}.$ On the other hand, 
the central extension ${k(X)^{*'}_{W_{\ov \O_X}^\ar}}$ splits, since it admits the following
natural section: an element $h\in k(X)^*$
takes $\Num({\cal A}_{\ov E_P,*}^{\,\,\Ar}(\ov \O_X))$ to 
$\Num({\cal A}_{\ov E_P,*}^{\,\,\Ar}(\ov \O_X(\mathrm{div}_\ar(h)))).$ 
That is, $h\in \Hom_\R\big(\Num({\cal A}_{\ov E_P,*}^{\,\,\Ar}(\ov \O_X)),
\,\Num({\cal A}_{\ov E_P,*}^{\,\,\Ar}(\ov \O_X(\mathrm{div}_\ar(h))))\big)$. 
This implies that the central extension  ${k(X)^{*'}_{W_{\ov \O_X}^\ar}}$ splits.
\end{proof}

\subsection{Reciprocity Symbols from Points and Curves}

Let $\pi:X\raw \mathrm{Spec}\,\O_F$ be a regular arithmetic surface defined over number field $F$. 
Let $C$ be a complete irreducible vertical curve or an irreducible horizontal curve, and $x$ a closed point 
of $C$. Set $H:=\mathrm{Frac}\big(\widehat\O_{X,x}\big)^*$.
Then with respect to the curve $C$ at $x$, resp. the closed point $x$ along $C$, there exists a natural central extension
$$0\to\R\raw\widehat H_{B_x}\raw H\to 1\eqno(A)$$
resp.
$$0\to\R\raw\widehat H_{\O_{K_{C,x}}}\raw H\to 1\eqno(B)$$
defined as follows:

\begin{itemize}
\item [(a)] As a set, its elements are given by pairs $(f,\alpha)$, resp. by pairs $(g,\beta)$, with
$f\in H$ and $\alpha\in [B_x|fB_x]_2$, resp. $g\in H$ and $\beta\in [\O_{K_{C,x}}|g\O_{K_{C,x}}]_2$;
\item[(b)]  For the group law, its multiplication is defined by
\end{itemize}
$$
(f_1,\alpha_1)\circ (f_2,\alpha_2)=(f_1f_2,\alpha_1f_1(\alpha_2)),\ \  \mathrm{resp}.\ \ 
(g_1,\beta_1)\circ (g_2,\beta_2)=(g_1g_2,\beta_1f_1(\beta_2)).
$$

\begin{defn}
For $f_1, f_2\in H$,
define a reciprocity symbol $[f_1,f_2]_{x,C}$, resp. $[f_1,f_s]_{C,x}$, by
$$
[f_1,f_2]_{x,C}:=[f_1',f_2']_A, \quad \mathrm{resp}.\quad [f_1,f_2]_{C,x}:=[f_1\mydprime ,f_2\mydprime ]_B,
$$
where
$f_1', f_2'$, resp. $f_1\mydprime , f_2\mydprime $, are the lifts of elements $f_1, f_2$ of $H$ via 
extension (A) to  $\dis{\widehat H_{B_x}}$, resp. 
via extension (B) to $\widehat H_{\O_{K_{C,x}}},$ and   $[f_1',f_2']_A$, resp. $[f_1\mydprime ,f_2\mydprime ]_B$, 
denotes the commutative of elements
$f_1', f_2'$, resp. $f_1\mydprime , f_2\mydprime $, in $\widehat H_{B_x}$, resp. in $\widehat H_{\O_{K_{C,x}}}$.
\end{defn}

We have the following

\begin{prop}
For $f_1,f_2\in H$, $$ [f_1,f_2]_{x,C} \cdot [f_1,f_2]_{C,x}=1.$$
\end{prop}

\begin{proof} 
For vertical curves, this proposition is simply Prop 13 of \cite{O1}. For horizontal curves, the same proof works. 
Indeed, similarly, for any two invertible sheaves $\L,\,\L'$ satisfying $\L\subset \L'$, there exists
a natural isomorphism
$$\begin{aligned}~
[B_x\otimes_{\widehat\O_{X,x}}\L \big| B_x \otimes_{\widehat\O_{X,x}} \L' ]_2
&\otimes_\R[\O_{K_{C,x}}\otimes_{\widehat\O_{X,x}} \L\big| \O_{K_{C,x}}\otimes_{\widehat\O_{X,x}} \L']_2\\
&\qquad \raw\Hom_\R\big(\Num({\cal A}_{C,x}(\L)),\Num({\cal A}_{C,x}(\L'))\big).\end{aligned}
$$
Consequently, if we introduce a third central extension $\widehat H_{B_x, \O_{K_{C,x}}}$ of $H$ by:

\begin{itemize}
\item [(a)] As a set, its elements are given by pairs $(f,\gamma)$ with
$f\in H$ and $\gamma\in [B_x|fB_x]_2\otimes_\R  [\O_{K_{C,x}}|g\O_{K_{C,x}}]_2$;
\item[(b)]  For the group law, its multiplication is defined by 
$$
(f_1,\gamma_1)\circ (f_2,\gamma_2)=(f_1f_2,\gamma_1f_1(\gamma_2)).
$$
\end{itemize}
\noindent
Then $\widehat H_{B_x, \O_{K_{C,x}}}$ splits.
\end{proof}

\subsection{Proof of Reciprocity Laws}
We may use same proofs as in  Theorem 2, resp. Theorem 3 of \cite{O1} to prove
our reciprocity law around a point, resp. along an irreducible complete curve, for arithmetic surfaces, 
\cite{O1} works on algebraic surfaces over finite fields. This is because in this two cases above, 
only finite part of arithmetic surface is involved. 

Thus, it suffices for us to prove reciprocity laws along horizontal curves for arithmetic surfaces. 

For this purpose, we set 
$$
M_{\ov E_P}^\ar:=\prod_{x\in E_P}K_{E_P,x}\times \prod_{Q\in X_F}K_{X_F,Q}\big|_P\widehat\otimes_\Q \R.
$$ 
Then for $f,g\in k(X)^*$, set $S_{E_P}$ to be the collection of places $x$ of $F$ such that $x$ is 
in the union of the support of $\mathrm{div}(f|_{E_P})$ and the support of $\mathrm{div}(g|_{E_P})$, 
and let $S_{X_F}$ to be the collection of points $Q$ of $X_F$ such that $Q$ is 
in the union of the support of $\mathrm{div}(f|_{X_F})$ and the support of $\mathrm{div}(g|_{X_F})$.
Let 
$$
M_{f,g}^\ar:=\prod_{x\in S_{E_P}}K_{E_P,x}\times \prod_{Q\in S_{X_F}}K_{X_F,Q}\big|_P\widehat\otimes_\Q \R
$$ 
be a combination of factors of $M_{\ov E_P}^\ar$, and define  $M^-_{f,g}$ to be its cofactor in 
$M_{\ov E_P}^\ar$  so that we have
$$
M_{\ov E_P}^\ar=M_{f,g}^\ar\times M_\ar^{f,g}.
$$

To go further,  we use the  central extension  $k(X)^*_{W_{\ov \O_X}^\ar}$ of $k(X)^*$ 
with respect to $W_{\ov \O_X}^\ar$. By its splitness proved in Prop. 10  of \S2.5, 
we have the associated reciprocity symbol $[*,*]_{W_{\ov \O_X}^\ar}$ vanishes. 
Similarly, we can also introduce the central extension  of 
$k(X)^*$ with respect to $M_{f,g}^\ar$ and $M^{f,g}_\ar$ to get first the groups $k(X)^*_{M_{f,g}^\ar}$ and 
$k(X)^*_{M^{f,g}_\ar}$ and hence the associated reciprocity symbols $[*,*]_{M_{f,g}^\ar}$ and 
$[*,*]_{M^{f,g}_\ar}$. Based on the techniques developed in \S2.2, it is rather direct to show that
$$
[*,*]_{W_{\ov \O_X}^\ar}=[*,*]_{M_{f,g}^\ar}\,+\,[*,*]_{M^{f,g}_\ar}.\eqno(*_3)
$$ 
Moreover, since $f,\,g$ keep $M^{f,g}_\ar$ unchanged, we have 
$$
[f,g]_{M^{f,g}_\ar}=0\qquad\forall f,g\in k(X)^*.
$$
Therefore, we have
$$
0=[f,g]_{W_{\ov \O_X}^\ar}=[f,g]_{M_{f,g}^\ar}\,+\,[f,g]_{M^{f,g}_\ar}=[f,g]_{M_{f,g}^\ar}.
$$
Indeed, for any factor subspace $M$ of $W_{\ov \O_X}^\ar$, we can construct its associated central extension 
$k(X)^*_M$ and hence obtaining the reciprocity pairing $[*,*]_M$ such that $[*,*]$ is additive, and, 
if $f,\,g\in k(X)^*$ keep $M$ stable, then, with the proof of  Prop 4(4) of \cite{O1}, 
we have $[f,g]_M=0$. Consequently, by $(*_3)$,
$$
0=\sum_{x\in E_P}[f,g]_{B_x}+\sum_{Q\in X_F}[f,g]_{O_{X_F}(Q)}.
$$
This is the abstract reciprocity law.

By Prop. 6 of \S3.6, we have $[f,g]_{B_x}=[f,g]_{C,x}=\nu_{C,x}(f,g)\log q_x$. Hence, to complete the proof, it suffices
to show that
$$
[f,g]_{O_{X_F}(Q)}=\log\frac{|f_0(P)^{\nu_P(g)}|}{|g_0(P)^{\nu_P(f)}|}.
$$

\subsection{End of Proof}

This now becomes very simple. Without loss of generality, we assume that there is only one place, a complex one,
and we view $k(X_F)\widehat\otimes_\Q\R$ as a subspace of $\C((u))$. Then we are led to the central extension
$\C((u))^*_{\C[[u]]}$, consisting of elements $(f,\alpha)$ where $f\in \C[[u]]$ and $\alpha\in [\,\C[[u]]\,\big|\,f\C[[u]]\,]$.
Construct a natural metric on $\C((u))$ using the standard metric on $\C$ and  $(u^i,u^j)=\delta_{ij}$. 
Accordingly, by the detailed calculations carried out in \S A.2.6, with the trivial metrics on $\C((u))$ 
and hence over various spaces used, we have
$$
[f|g]_{\C[[u]]}=-\log\frac{|f_0(P)|^{\nu_P(g)}}{|g_0(P)|^{\nu_P(f)}}.
$$
This then ends our proof.

\vskip 0.20cm
\noindent
{\it Remark.} Note that the appearances of $-\log|f(P)|$ is not surprising. Indeed, 
when we numerate spaces $W_{\ov\L}^\ar$, we use the spaces $W_{\ov \L,s}^{\ar}$ defined by
{\footnotesize$$
W_{\ov \L,s}^{\ar}=\Big(\prod_{x\in E_P}\frak m_{E_P,x}^{-n_x}((u))\Big)
\times \Big(\prod_{Q\in X_F}\O_{X_F}(Q)^{\otimes m_Q}\big|_P((u))\,\widehat\otimes_\Q\R\Big)
\times {\bf e}\Big(\int_{X_\infty}-\log\|s\|d\mu\Big).
$$}
It is the additional factor ${\bf e}\Big(\int_{X_\infty}-\log\|s\|d\mu\Big)$ which makes the whole intersection 
theory and hence our construction work well.
For example, the canonical numeration $\num_0$ for the space $W_{\ov \L,s}^{\ar}$ is given by
$$
\begin{aligned}
&\num_0\Big(\Big(\prod_{x\in E_P}\frak m_{E_P,x}^{-n_x}\Big)
\times \Big(\prod_{Q\in X_F}\O_{X_F}(Q)^{\otimes m_Q}\big|_P\,\widehat\otimes_\Q\R\Big)
\times {\bf e}\Big(\int_{X_\infty}-\log\|s\|d\mu\Big)\Big)\\
=&\prod_{x\in E_P}q_x^{n_x}\prod_{Q\in X_F}G_\infty(P,\di(s_F))\cdot {\bf e}\Big(\int_{X_\infty}-\log\|s\|d\mu\Big).
\end{aligned}
$$
In particular, when $\ov \L=\ov \O_X(\mathrm{div}_\ar(f))$, then 
$\num_0\big(W_{\\ov \O_X(\mathrm{div}_\ar(f)),f}^{\ar}\big)=0$ by the product formula. Consequently, 
for different choices of rational sections $s$ and $s'$ of $\L$,
we have canonical isometry
$$
\Num\big(W_{\ov \L,s}^{\ar}\big)\cong \Num\big(W_{\ov \L,s'}^{\ar}\big).
$$
In particular, it makes  sense to introduce the space $W_{\ov \L}^\ar$. 
As a direct consequence, for rational function $f\in k(X_F)^*$, we obtain this  crucial term 
$-\sum_{\sigma|\infty}\log|f_\sigma|_\sigma$ and hence the 
reciprocity symbol at infinite.

\vskip 8.80cm
K. Sugahara \& L. Weng

Graduate School of Mathematics,
 
Kyushu University,
 
Fukuoka, 819-0395,
 
Japan
 
E-mails: k-sugahara@math.kyushu-u.ac.jp,
 
\hskip 2.30cm  weng@math.kyushu-u.ac.jp

\end{document}